\documentclass[a4paper]{amsart}

\usepackage{graphicx}
\usepackage{amsmath}
\usepackage{amsthm}
\usepackage{amssymb}
\usepackage{enumerate}
\usepackage{mathrsfs} 
\usepackage{hyperref}
\usepackage[nooneline, figbotcap]{subfigure}

\theoremstyle{plain}
  \newtheorem{theorem}{Theorem}[section]
  
  \newtheorem{proposition}[theorem]{Proposition}
  \newtheorem{algorithm}[theorem]{Algorithm}

\theoremstyle{definition}
  \newtheorem{definition}[theorem]{Definition}

\theoremstyle{remark}

\numberwithin{equation}{section}

\def\S{\Sigma}
\def\D{\mathcal{D}}

\subfiguretopcaptrue

\begin{document}

\title{Representing knots by filling Dehn spheres}

\author{\'Alvaro Lozano Rojo}
\address{Centro Universitario de la Defensa Zaragoza, Academia General Militar
Carretera de Huesca s/n. 50090 Zaragoza, Spain --- IUMA, Universidad de
Zaragoza}
\email{alvarolozano@unizar.es}

\thanks{Both authors have been supported by the European Social Fund
and Diputaci\'on General
de Arag\'on (Grant E15 Geometr{\'\i}a), and by Spanish Government's
\emph{Programa Estatal de Fomento de la  Investigaci\'on Cient\'ifica y
T\'ecnica de
Excelencia} (research projects MTM2013-46337-C2 and MTM2013-45710-C2 for the
first and second author respectively)
}

\author{Rub\'en Vigara Benito}
\address{Centro Universitario de la Defensa Zaragoza, Academia General Militar
Carretera de Huesca s/n. 50090 Zaragoza, Spain --- IUMA, Universidad de
Zaragoza}
\email{rvigara@unizar.es}




\begin{abstract}
  We prove that any knot or link in any $3$-manifold can be nicely decomposed 
  (\emph{splitted}) by a filling Dehn sphere. This has interesting 
  consequences in the study of branched coverings over knots and links. We give
  an 
  algorithm for computing Johansson diagrams of filling Dehn surfaces out from 
  coverings of $3$-manifolds branched over knots or links.
\end{abstract}

\maketitle

\section{Introduction}\label{sec:intro}

Through the whole paper all $3$-manifolds are assumed to be orientable and 
closed, that is, compact connected and without boundary. On the contrary, 
surfaces are assumed to be orientable, compact and without boundary, but they
could be non-connected. $M$ generically denotes a $3$-manifold and $S$ a
surface. Although all the constructions can be adapted to work in the
topological or PL categories, we work in the smooth category: manifolds have a
differentiable 
structure and all maps are assumed to be smooth.

Let $M$ be a $3$-manifold.

A \emph{Dehn surface} in $M$~\cite{Papa} is a surface (the \emph{domain} of
$\S$) immersed in $M$ in general
position: with only double curve and triple point singularities. 
The Dehn surface $\S\subset M$ \emph{fills} $M$~\cite{Montesinos} if it defines
a cell-decomposition of $M$ in which the $0$-skeleton is the set $T(\S)$
of triple points of $\S$, the $1$-skeleton is the set 
$S(\S)$ of double and triple points of $\S$, and the $2$-skeleton is $\S$ 
itself (the notation $T(\S),S(\S)$ is similar to that introduced
in~\cite{Shima}). If $\Sigma$ is a Dehn surface in $M$, a connected component
of 
$S(\S)-T(\S)$ is an \emph{edge} of $\S$, a connected component of 
$\Sigma-S(\Sigma)$ is a \emph{face} of $\Sigma$, and a connected component of 
$M-\Sigma$ is a \emph{region} of $\Sigma$. The Dehn surface $\S$ fills $M$ if 
and only if all its edges, faces and regions are open $1$, $2$ or 
$3$-dimensional disks respectively.

Following ideas of~\cite{Haken1}, in~\cite{Montesinos} it is proved that every 
$3$-manifold has a filling Dehn sphere (see also~\cite{Anewproof}, and 
specially~\cite{Amendola09} where an extremely short and elegant proof of this 
result can be found), and filling Dehn spheres and their Johansson diagrams 
are proposed as a suitable way for representing all $3$-manifolds. A weaker 
version of  filling Dehn surfaces are the \emph{quasi-filling Dehn surfaces} 
defined in~\cite{Amendola02}, which are Dehn surfaces whose complementary set 
in $M$ is a disjoint union of open $3$-balls. In~\cite{FennRourke} it is proved 
that every $3$-manifold has a quasi-filling Dehn sphere. 
Since~\cite{Montesinos}, some other papers have been appeared on this subject 
(cf.~\cite{Amendola09,Amendola02,racsam,spectrum,Anewproof,RHomotopies,RFenn}), 
applying filling Dehn surfaces to different aspects of $3$-manifold topology. 
This is the first of a series of papers where we will give some tools for 
applying filling Dehn surfaces to knot theory.

Let $K$ be a tame knot or link in $M$.

\begin{definition}
  The filling Dehn surface $\S$ of $M$ \emph{splits} $K$ if:
  \begin{enumerate}

    \item $K$ intersects $\S$ transversely in a finite set of non-singular 
      points of $\S$;

    \item $K-\S$ is a disjoint union of open arcs;

    \item for each region $R$ of $\S$, if the intersection $R\cap K$ is 
      non-empty it is exactly one arc, unknotted in $R$;

    \item for each face $F$ of $\S$, the intersection $F\cap K$ contains at most
      one point.
      
  \end{enumerate}
\end{definition}

\begin{theorem}
  \label{thm:knots-can-be-splitted}
  Every link $K$ in a $3$-manifold $M$ can be splitted by a filling Dehn sphere 
  of $M$.

  Moreover, this filling Dehn sphere can be chosen such that it intersects 
  exactly twice each connected component of $K$.
\end{theorem}

\begin{proof}
  The proof is a direct consequence of the \emph{inflating of triangulations},
  construction introduced
  in~\cite{peazolibro,RHomotopies,tesis}.
  
  Take a smooth triangulation $T$ of $M$ \cite{WhiteheadSmoothTriang} such 
  that $K$ is simplicial in $T$. This can be done because $K$ is tame. 
  In~\cite{peazolibro,RHomotopies,tesis} it is shown how we can ``inflate''
  the 
  $0$,$1$ and $2$-dimensional simplices of $T$ in order to obtain a filling
  collection of 
  spheres in $M$, and how we can connect them until we obtain a unique filing
  Dehn sphere $\S_T$ of $M$. By construction it is easy to check that $\S_T$ 
  splits $K$. In fact $\S_T$ splits every link which is a subcomplex of $T$.

  For the second statement, we modify $\S_T$ slightly around each connected 
  component $K_i$ of $K$. Instead of introducing one sphere for each vertex and 
  for each edge of the triangulation $T$ belonging to $K_i$, we take a unique 
  Dehn sphere containing $K_i$ as in Figure~\ref{fig:inflated_trefoil}, while 
  the rest of $\S_T$ remains the same. With this modification, $\S_T$ 
  intersects $K_i$ twice. By repeating the same operation around each connected 
  component of $K$ we get the desired result.
\end{proof}

\begin{figure}
  \centering
  \includegraphics[width=0.6\textwidth]{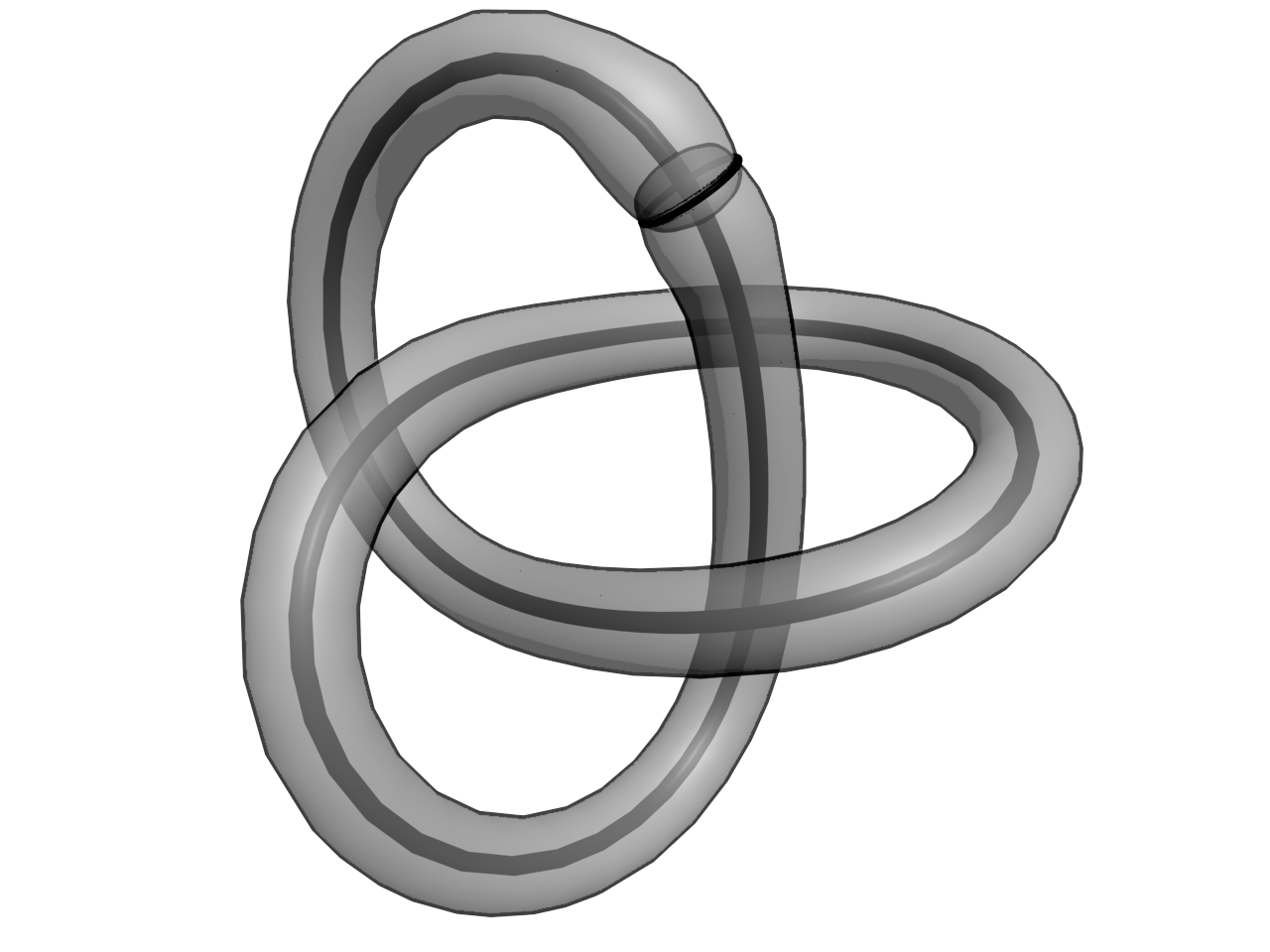}
  \caption{A Dehn sphere containing the trefoil knot}
  \label{fig:inflated_trefoil}
\end{figure}

If the filling Dehn
surface $\S$ splits $K$, when we draw the points $P_1,P_2,\ldots,P_k$ of 
$\S\cap K$ on the domain $S$ of $\S$ 
together with the Johansson diagram $\D$ of
$\S$, we
obtain the \emph{Johansson diagram} $(\D,\{P_1,P_2,\ldots,P_k\})$ of the knot or
link $K$.
Because the manifold $M$ can be built up from
$\D$~\cite{peazolibro,tesis} and the arcs of $\S-K$
are unknotted, we have:
\begin{proposition}
 It is
possible to recover the pair $(M,K)$ from the Johansson diagram
$(\D,\{P_1,P_2,\ldots,P_k\})$ of the knot $K$.
\end{proposition}

If $K$ is a knot, we will say
that a filling Dehn sphere $\S$ that splits $K$ is \emph{diametral} for $K$ or
that \emph{diametrically splits} $K$ if it intersects $K$ exactly twice. In
this case, if $\D$ is the Johansson diagram of $\S$ and $\S\cap K=\{P_1,P_2\}$,
the pair $(\D,\{P_1,P_2\})$ is a \emph{diametral Johansson diagram} of the
knot $K$. As it
will be shown in Section~\ref{sec:unknot}, diametral filling Dehn spheres are
specially useful for studying the branched covers of $M$ with branching set $K$.
For a given knot $K$, it should be interesting to have an algorithm that
provides its simplest diametrically splitting filling Dehn sphere. Of course,
before that we would have to define ``simplest", perhaps in terms of the
number
of triple points and/or double curves of the filling Dehn sphere.
 
Algorithms for computing the Johansson diagram of a filling Dehn sphere of $M$
from a Heegard diagram of $M$ are provided in~\cite{Montesinos,Anewproof}.
In Section~\ref{sec:splitting} we outline an algorithm
for computing such Johansson diagrams from branched covers of $M$ with
branching set a tame knot or link $K$ in $M$. As an application of
this algorithm,
we apply it in Section~\ref{sec:unknot} to the coverings of $S^3$ branched over
the unknot.

\section{Dehn surfaces and their Johansson's diagrams}
\label{sec:Dehn-surfaces-Johansson-diagrams} 

We refer to~\cite{RHomotopies,racsam,spectrum} for more detailed definitions
about Dehn sufaces and Johansson diagrams, and 
to~\cite{BersteinEdmonds,MontesinosBranched}
as a basic reference about
branched coverings.

A subset $\S\subset M$ is a Dehn surface in $M$ if there 
exists a surface $S$ and a general
position immersion $f:S\rightarrow M$ such 
that $\Sigma=f\left( S\right)$. In this situation we say that $S$ is the 
\emph{domain} of $\S$ and that $f$ \emph{parametrizes} $\Sigma$. If $S$ is a 
$2$-sphere, a torus,\ldots  then $\Sigma$ is a \emph{Dehn sphere}, a 
\emph{Dehn torus},\ldots respectively. 

Let $\S$ be a Dehn surface in $M$ and consider a parametrization $f:S\to M$ 
of $\S$. The preimage under $f$ in $S$ of the 
singularity set $S(\S)$ of $\Sigma$, together with the information about how
its 
points become identified by $f$ in $\S$ is the \emph{Johansson diagram} $\D$ 
of $\S$ (see~\cite{Johansson1,Montesinos}).

Because $S$ is compact and without boundary, double curves are closed and there
is a finite number
of them. The number of triple points is also finite. Because $S$ and $M$ are
orientable, the preimage under $f$ of a double curve of $\S$ is the union of two
different closed curves in $S$, and we will say that these two curves are
\emph{sister curves} of $\D$. Thus, the Johansson diagram of $\S$ is composed by
an even number of different closed curves in $S$, and we will  identify $\D$
with  the  set of different  curves that compose it. For any curve $\alpha\in\D$
we denote by $\tau \alpha$ the sister curve of $\alpha$ in $\D$. This defines a
free involution $\tau:\D\rightarrow\D$, the \emph{sistering} of $\D$, that sends
each curve of $\D$ into its sister curve of $\D$.

If we are given an \emph{abstract diagram}, i.e., an even collection of
curves 
in $S$ coherently identified in pairs, it is possible to know if this 
abstract diagram is \emph{realizable}, that is, if it is actually the Johansson
diagram 
of a Dehn surface in a $3$-manifold (see~\cite{Johansson1,Johansson2,tesis}). 
It is also possible to know if the abstract diagram is \emph{filling}: if it
is 
the Johansson diagram of a filling Dehn surface of a $3$-manifold 
(see~\cite{tesis}). If $\S$ fills $M$, it is possible to build $M$ out of the
Johansson diagram of $\S$. Thus, filling Johansson diagrams represent all
closed, orientable $3$-manifolds.
When a diagram $\D$ in $S$ is not realizable, the
quotient space of $S$ under the equivalence relation defined by the diagram is
something very close to a Dehn surface: it is a 2-dimensional complex with
simple, double and triple points, but it cannot be embedded in any $3$-manifold.
We reserve the name \emph{pseudo Dehn surface} for these objects. Many
constructions about Dehn surfaces, as the presentation of their fundamental
group given in~\cite{spectrum}, for example, are also valid
for pseudo Dehn surfaces.

A special case of filling Dehn surfaces is when the domain $S$ of the
filling Dehn surface $\S$ is a disjoint union of $2$-spheres. In this case, we
say that $\S$ is a \emph{filling collection of spheres}. Starting from a
filling collection of surfaces $\S$ in $M$ we can always reduce the number of
connected components of the domain $S$ of $\S$ without loosing the filling
property by applying surgery modifications (see
\cite{Banchoff,tesis,peazolibro}). Eventually, starting from a filling
collection of surfaces in $M$
we can always obtain a filling Dehn sphere of $M$.

\section{Splitted knots and branched coverings}\label{sec:splitting}

Let $K$ be a tame knot or link in $M$.

Assume that the filling Dehn surface $\S$ of $M$ splits $K$. Let 
$R_1,R_2,\ldots,R_m$ be all the different regions of $\S$ whose intersection 
with $K$ is empty, and take one point $Q_i$ in $R_i$ for $i=1,2,\ldots,m$. 
Then, $\S-K$ is a strong deformation retract of
\[
  M-(K\cup{Q_1,Q_2,\ldots,Q_m})\,,
\]
and therefore,

\begin{proposition}
 The fundamental group of $M-K$ is isomorphic to the fundamental 
 group of $\S-K$. \qed
\end{proposition}

Our main interest on knots splitted by filling Dehn surfaces is due to the
following theorem.

\begin{theorem}
  \label{thm:lift-of-filling-surfaces}
  If $p:\hat{M}\to M$ is a branched covering with downstairs branching set $K$,
  then $\hat{\S}=p^{-1}(\S)$ is a filling Dehn surface of $\hat{M}$.
\end{theorem}

\begin{proof}
  As $\S\cap K$ contains no singular point of $\S$, $\hat{\S}$ is a Dehn 
  surface in $\hat{M}$ and all the edges of $\hat{\S}$ are open $1$-arcs.
  
  If a face $\hat{F}$ of $\hat{\Sigma}$ does not intersect the lift
  $\hat{K}:=p^{-1}(K)$ of the link $K$, then $\hat{F}$ is a regular covering of
  a face of $\Sigma$. In the other case, $\hat{F}$ is a branched covering of a
  face F with branching set the unique point of $F\cap K$
  In both cases $\hat{F}$ must be an open $2$-disk. In the 
  same way, a region $\hat{R}$ of $\hat{\S}$ is a regular covering space of a 
  region of $\S$, if $\hat{R}$ does not intersect $\hat{K}$, or it is a 
  branched covering space of a region $R$ of $\S$ whose branching set in $R$ is 
  the unknotted arc $R\cap K$. In both cases, $\hat{R}$ must be an open 
  $3$-ball. This implies that $\hat{\S}$ fills $\hat{M}$.
\end{proof}

In the previous theorem the filling Dehn surfaces $\S$
and $\hat{\S}$ could have different domains. 
Let $S$ be the domain of $\S$, let $f:S\to M$ be a parametrization of $\S$ and
let $f^{-1}(K)$ be the set of points $\{P_1,P_2,\ldots,P_k\}$. We will denote
by $M_K$ and $S_K$ the sets
$$M-K\qquad \text{and}\qquad S-\{P_1,P_2,\ldots,P_k\}\,,$$
respectively. Take
a non-singular point
$x$ of $\S$ as the base point of the fundamental group $\pi_K:=\pi_1(M_K,x)$ of
$M_K$. We
denote also by $x$ the preimage of $x$ under $f$, and we choose it as base
point of the fundamental group $\pi_{S_K}:=\pi_1(S_K,x)$ of
$S_K$. The map $f_K$ defined as
\[
  f_K=f|_{S_K}:S_K\to M_K
\]
induces an homomorphism
\[
  \left(f_K\right)_*: \pi_{S_K}\to \pi_K\,.
\]
On the other hand, the branched covering $p:\hat{M}\to M$ has an
associated monodromy homomorphism
$\rho:\pi_K\to \Omega_n$ into the group $\Omega_n$ of permutations of $n$
elements, with $n=1,2,\ldots$.  

In~\cite{spectrum}, in order to give a presentation of the fundamental group of
a Dehn surface,
it is made a detailed study of unbranched coverings of Dehn surfaces. In
particular, for a covering
$p:\hat{\S}\to \S$ it is shown how to construct the domain $\hat{S}$ of
$\hat{\S}$
and the Johansson diagram of $\hat{\S}$ on $\hat{S}$. Although the results given
in 
\cite{spectrum} are stated for orientable Dehn surfaces of genus $g$ (\emph{Dehn
$g$-tori}),
the construction is also valid for more general surfaces, as the punctured
surface $\S_K$.

\begin{theorem}
  \label{thm:domain-of-lifted-surface}
  The domain $\hat{S}$ of $\hat{\S}$ is a branched covering space 
  $p_S:\hat{S}\to S$ with downstairs branching set $\{P_1,P_2,\ldots,P_k\}$ and 
  monodromy homomorphism $\rho\circ(f_K)_*$.
\end{theorem}

\begin{proof}
  The proof follows from~\cite[Sec. 4]{spectrum}.
\end{proof}

The branched covering $p_S:\hat{S}\to S$ is called the \emph{domain branched 
covering} of $p$ for $\S$ and the monodromy homomorphism 
$\rho_S:=\rho\circ(f_K)_*$ is called the \emph{domain monodromy} of 
$p$ for $\S$.

The previous results provide an algorithm for obtaining the Johansson diagram 
of a filling Dehn sphere of the coverings of $M$ branched over $K$. This 
algorithm can be summarized as follows:

\begin{algorithm}
  \label{algorithm}
  \begin{enumerate}
    
    \item Find a filling Dehn surface $\S_K$ in $M$ that splits $K$;

    \item draw the Johansson diagram $\D_K$ of $\S_K$ and mark the points
      $P_1,P_2,\ldots,P_r$ of $\S_K\cap K$ on it;

    \item find a presentation of the knot group $\pi_K$ of $K$ in terms of 
      $\S_K$ (cf.~\cite{spectrum});

    \item for each transitive representation $\rho_K$ of $\pi_K$ into a 
      permutation group $\Omega_n$, find the domain monodromy $\rho_{S}$ of 
      $\pi_{S_K}$ into $\Omega_n$;

    \item build the branched covering of $p_S:\hat{S}\to S$ with branching
      set $\{P_1,\ldots,P_r\}$ and monodromy $\rho_{S}$, and lift the 
      diagram $\D_K$ to a diagram $\hat{\D}$ on $\hat{S}$;
      
    \item by identifying the curves of $\hat{\D}$ using the
images under the representation $\rho$ of the loops dual to the
to the curves of $\D$  we obtain the Johansson diagram of a filling Dehn
 of the 
      branched covering $p^3:\hat{M} \to M$ with branching set $K$ and 
      monodromy $\rho_K$.

  \end{enumerate} 
\end{algorithm}

As the simpler filling Dehn surfaces are filling Dehn spheres, we 
want to use this algoritm to obtain the Johansson diagram of a
filling Dehn sphere of $\hat{M}$.

\begin{definition}
  A covering $p:\hat{M}\to M$ branched over the knot $K$ is \emph{locally
cyclic} if  the 
  monodromy map $\rho$ sends knot meridians onto $n$-cycles, where $n$ is the 
  number of sheets of the covering. This is equivalent to say that 
  $p:p^{-1}(K)\to K$ is a homeomorphism.
\end{definition}

\begin{theorem}
  \label{thm:diametral-lifting-of-spheres-are-spheres}
  In the hypotheses of Theorem~\ref{thm:lift-of-filling-surfaces}, if $K$ is a 
  knot and $\S$ is a filling Dehn sphere in $M$ which is diametral for $K$,
then 
  $\hat{\S}$ is a filling collection of spheres, and it is a Dehn sphere if and 
  only if $p$ is locally cyclic.
\end{theorem}

\begin{proof}
As $\S$ is diametral for $K$, by Theorem~\ref{thm:domain-of-lifted-surface} the
domain $\hat{S}$ of $\hat{\S}$
is a branched covering space of the domain $S$ of $\S$, which is a $2$-sphere,
with two points as branching set.
This implies that $\hat{S}$ is a disjoint union of $2$-spheres, and by Theorem
\ref{thm:lift-of-filling-surfaces},
$\hat{\S}$ is a filling collection of spheres. The Dehn surface $\hat{\S}$ is a
Dehn sphere
if and only if its domain $\hat{S}$ is connected, and this occurs if and only if
the domain monodromy 
$\rho_S$ acts transitively. This is equivalent
to the fact that
the branched covering $p$ is locally cyclic.
\end{proof}

Therefore, after combining Algorithm~\ref{algorithm} and surgery operations in
filling collections of spheres, we can eventually obtain the Johansson diagram
of a
filling Dehn sphere of $\hat{M}$.

\section{Example: Johansson's sphere and the unknot}\label{sec:unknot}

\begin{figure}\label{fig:Johansson_sphere}
\centering
\subfigure[]{
  \includegraphics{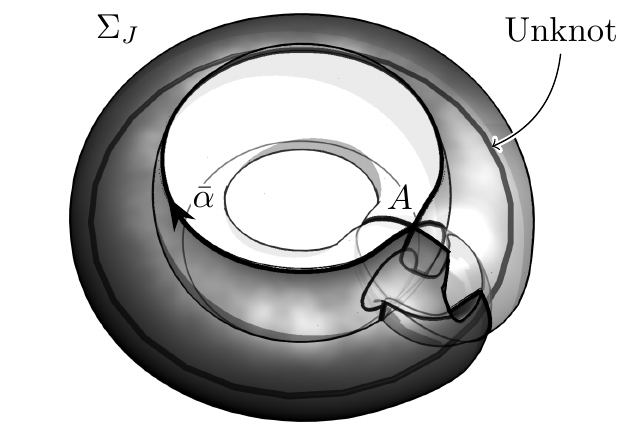}
  \label{fig:Johansson_sphere_01}%
}%
\\
\subfigure[]{%
  \includegraphics{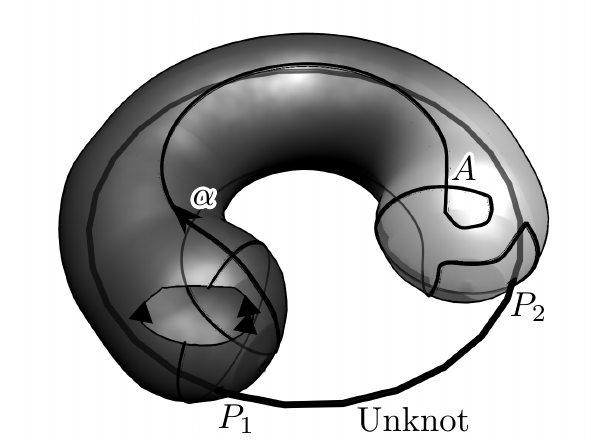}
  \label{fig:Johansson_sphere_diagram_donut}%
}\hfill%
\subfigure[]{
  \includegraphics{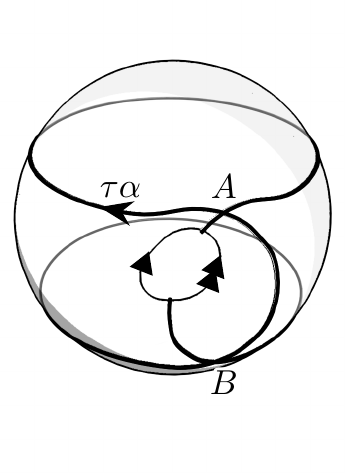}
  \label{fig:Johansson_sphere_diagram_ball}%
}
\caption{Johansson's sphere splitting the unknot}
\end{figure}

The Dehn sphere $\S_J$ depicted in Figure~\ref{fig:Johansson_sphere_01} 
is called \emph{Johansson's
sphere} in~\cite{tesis} because its Johansson diagram appears
in~\cite{Johansson1}. In the Johansson diagram the two sister curves must be
identified as indicated by the arrows. The Dehn sphere $\S_J$ can be obtained
using an algorithm introduced in~\cite{Anewproof},
after connecting by surgery an immersed
torus-like sphere (Figure~\ref{fig:pre-Johansson_sphere_diagram_donut}) with an
embedded sphere (Figure~\ref{fig:pre-Johansson_sphere_diagram_ball})
intersecting as in Figure~\ref{fig:pre-Johansson_sphere_01}.
According to~\cite{Shima} and~\cite{tesis}, it is one of
the three unique filling Dehn spheres of $S^3$ with only two triple points. If
we consider the unknot $K$ intersecting $\S_J$ as in the same 
Figure~\ref{fig:Johansson_sphere_01} , it is clear
that $\S_J$ splits $K$, and it can be checked using
Example 7.1 of~\cite{Anewproof} that the two points $P_1,P_2$ of $\S_J\cap K$
correspond to the points also labelled by $P_1,P_2$ in the diagram of 
Figure~\ref{fig:Johansson_diagram}.
\begin{figure}
\centering
\subfigure[]{%
  \includegraphics{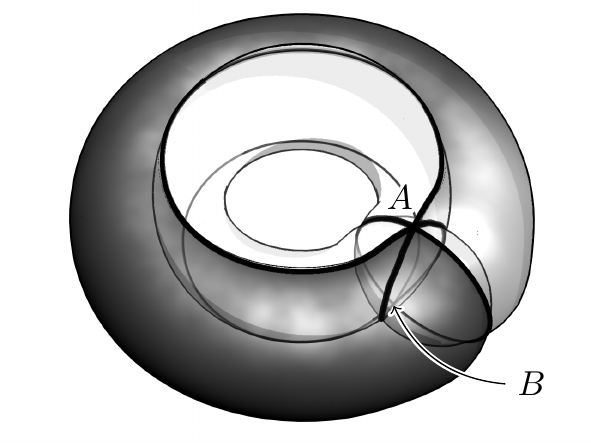}
  \label{fig:pre-Johansson_sphere_01}}%
  \\
\subfigure[]{%
  \includegraphics{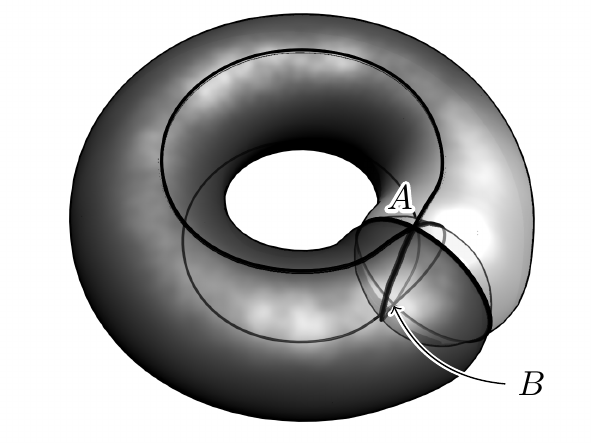}
  \label{fig:pre-Johansson_sphere_diagram_donut}}%
  \hfill%
\subfigure[]{%
  \includegraphics{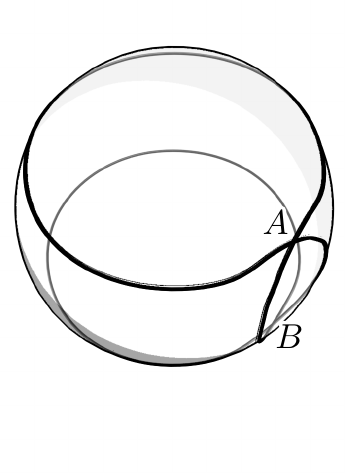}
    \label{fig:pre-Johansson_sphere_diagram_ball}%
  }
  \caption{Constructing Johansson's sphere}
  \label{fig:pre-Johansson_sphere}
\end{figure}

By~\cite{spectrum}, the fundamental group of $\S_J-K$ based at $x$ is 
generated by the loops $m$ and $a$, where:
\begin{itemize}
  
  \item $m$ is (the image in $S^3$ of) the generator of the fundamental group 
    of $S^2-\{P_1,P_2\}$ depicted in Figure~\ref{fig:Johansson_diagram};

  \item  the loop $a$ is \emph{dual} to the curve $\alpha$ of the diagram $\D$: 
    it is the loop in $\S_J$ composed by the product of paths 
    $\lambda_\alpha*\lambda_{\tau \alpha}^{-1}$, where $\lambda_\alpha$ and 
    $\lambda_{\tau \alpha}$ are paths joining $x$ with related points on 
    $\alpha$ and $\tau\alpha$ (Figure~\ref{fig:Johansson_diagram}).	 

\end{itemize}
It must be noted that the loop $m$ in $\S_J$ is homotopic to a meridian of $K$.
Using~\cite{spectrum}, it can be checked also that the element $a$ is trivial
in $\pi_K$ and that $\pi_K$ is the infinite cyclic group generated by $m$.
\begin{figure}
  \centering
  \subfigure[]{
    \includegraphics{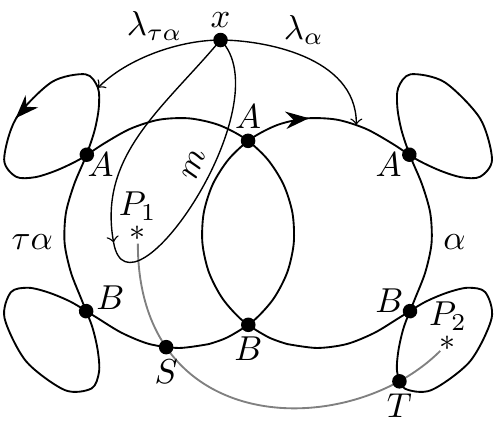}
    \label{fig:Johansson_diagram}
  }%
  \hfill%
  \subfigure[]{
    \includegraphics{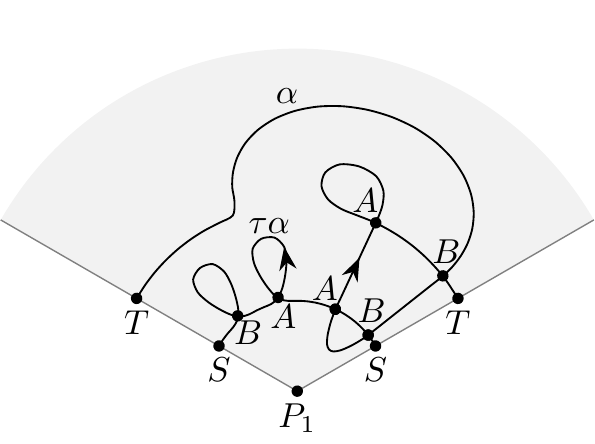}
    \label{fig:abanico}
  }
  \caption{Johansson's diagram of the unknot}
  \label{fig:Johansson_diagram_abanico}
\end{figure}

As $\pi_K$ is cyclic, the only coverings of $S^3$ branched over $K$ are the
cyclic ones, and it is well known that the unique manifold that covers $S^3$
branching over $K$ is again $S^3$. Assume that $p:S^3\to S^3$ is the
$n$-fold covering branched over $K$. By Theorems
\ref{thm:lift-of-filling-surfaces} and 
\ref{thm:diametral-lifting-of-spheres-are-spheres}, $\hat{\S}_J=p^{-1}(\S_J)$ is
a filling Dehn sphere of $S^3$. By Theorem~\ref{thm:domain-of-lifted-surface}
and~\cite{spectrum}, the Johansson diagram $\hat{\D}_J$ of $\hat{\S}_J$ is
obtained lifting $\D_J$ to the $n$-fold cover of $S^2$ branched over
$\{P_1,P_2\}$, and the sistering of the curves of $\hat{\S}_J$ is determined
by the image of the dual loops of the curves of $\D_J$ under the monodromy
homomorphism induced by $p$. If we cut the $2$-sphere of 
Figure~\ref{fig:Johansson_diagram} along an arc joining $P_1$ and $P_2$ and
after that we send the point $P_2$ through infinity, we obtain a \emph{fan} as
this depicted in Figure~\ref{fig:abanico}. By pasting ciclically $n$ copies
of this fan we obtain the Johansson diagram of $\hat{\S}_J$.
We have depicted the diagrams that
arise in this
situation for the cases $n=2,3$ in Figure~\ref{fig:liftings}. Because $a$ is
trivial in $\pi_K$, the sistering of the curves of the diagram must be as
indicated in this figure. In all cases we obtain a Johansson diagram with
just two curves, both with self-intersections, and this implies that the Dehn
spheres $\hat{\S}_J$ are all simply connected. If we had chosen $K$ in a
different relative position with respect to $\S_J$, we would have obtained a
different family of Johansson diagrams of $S^3$ from these branched coverings.

Of course, the cyclic coverings of $S^3$ branched over the unknot might not be
very interesting, but \emph{exactly the same construction
can be
done for the cyclic and locally cyclic branched coverings of any knot $K$}, once
a diametral Johansson diagram of $K$ have been obtained.
\begin{figure}
  \label{fig:liftings}
  \centering
  \subfigure[]{%
    \includegraphics{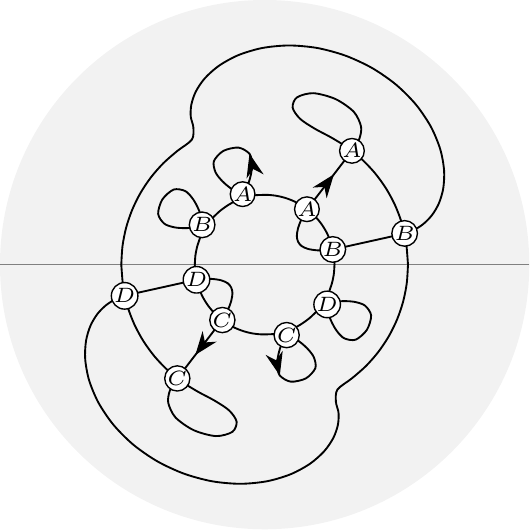}
    \label{fig:liftings2}%
  }%
  \hfill%
  \subfigure[]{%
    \includegraphics{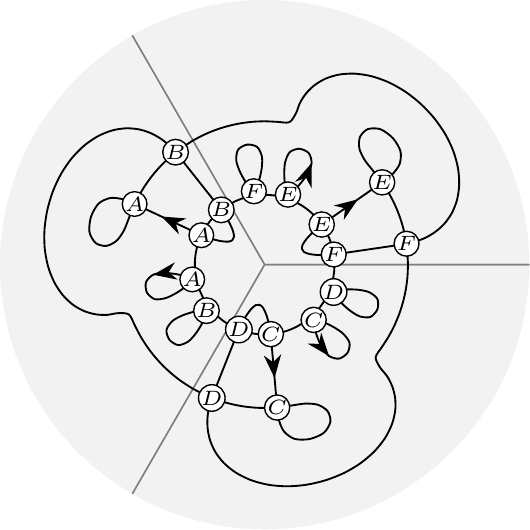}
    \label{fig:liftings3}%
  }
  \caption{Lifting $\S_J$ to cyclic branched coverings over the unknot}
\end{figure}

\section*{Acknowledgements}
The authors want to thank Professors J.M. Montesinos and M.T. Lozano for their
valuable advices about the content of this paper.

\bibliographystyle{amsplain}

\end{document}